\documentclass{amsart}

\usepackage{amsfonts,amsmath,amscd,amssymb,xspace}
\usepackage{longtable}
\usepackage{xcolor}
\usepackage{latexsym,mathtools}

\newcommand{\bQ}{{\mathbb Q}}
\newcommand{\bC}{{\mathbb C}}
\newcommand{\bH}{{\mathbb H}}
\newcommand{\bN}{{\mathbb N}}
\newcommand{\bR}{{\mathbb R}}
\newcommand{\bF}{{\mathbb F}}

\newcommand{\Irr}{{{\mathrm I}{\mathrm r}{\mathrm r} (G) }}
\newcommand{\Irh}{{{\mathrm I}{\mathrm r}{\mathrm r} (\widehat{G}) }}
\newcommand{\Clg}{{{\mathrm C}{\mathrm l} (G) }}
\newcommand{\Fun}{{{\mathrm F}{\mathrm u}{\mathrm n} (G/ \! / G, \bC)}}

\newcommand{\bCG}{{\bC \!\ast \! \widehat{G}}}
\newcommand{\bCC}{{\bC \!\ast \! (G\times C_2)}}

\newcommand{\cA}{{{\mathcal A}}}
\newcommand{\cB}{{{\mathcal B}}}
\newcommand{\cC}{{{\mathcal C}}}
\newcommand{\cD}{{{\mathcal D}}}
\newcommand{\cF}{{{\mathcal F}}}

\newtheorem{thm}{Theorem}[section]
\newtheorem{lemma}[thm]{Lemma}
\newtheorem{cor}[thm]{Corollary}
\newtheorem{con}[thm]{Conjecture}
\newtheorem*{BDP}{Brauer-Dyson Problem}

\title{Brauer's 14th Problem and Dyson's Tenfold Way}
\author{Dmitriy Rumynin}
\address{Department of Mathematics, University of Warwick, 
Coventry, CV4 7AL, UK}
\email{D.Rumynin@warwick.ac.uk}
\author{James Taylor}
\address{Dipartimento di Matematica, Universit\`{a} degli Studi di Padova, \newline Via Trieste 63, 35131 Padova, Italy}
\email{james.taylor@math.unipd.it}
\date{March 27, 2025}

\begin{document}
\begin{abstract}
  We consider Brauer's 14th Problem in the context of {\em Real} structures on finite groups
and their antilinear representations.
The problem is to count the number of characters of each different type using ``group theory''.
While Brauer's original problem deals only with three types (real, complex and quaternionic),
here we consider the ten types coming from Dyson's tenfold way.  
\end{abstract}

\maketitle

Let $G$ be a finite group. The set of complex irreducible characters $\Irr$ is split into three subsets:
each character could be of real, complex or quaternionic type.
The type of the character corresponds to the endomorphism algebra of the corresponding simple $\bR G$-module.
A reader could consult \cite[Appendix]{Li}
for a concise summary of the properties of the correspondence between $\Irr$ and the set of simple $\bR G$-modules.

Since we are not interested in the characters themselves,
it is convenient to keep replacing the set of characters with the multisets of their dimensions, e.g., 
\[
\cD_G \coloneqq [ \; \chi (1) \;\mid\; \chi\in \Irr \;], \ \ 
\cD_\bF \coloneqq [ \; \chi (1) \;\mid\; \chi\in \Irr \mbox{ and } \chi \mbox{ is of type } \bF\; ],
\]
where $\bF\in \{\bR, \bC, \bH\}$.
Let $N_\bF \coloneqq | \cD_\bF |$.

The set of conjugacy classes $\Clg$ is split only into two subsets. A class $K$ is of type $1$ or {real} if $K^{-1}=K$.
Correspondingly, it is of type $2$ or {not real} if  $K^{-1}\neq K$. Let $C'_1$ and $C'_2$ be the number of each type of classes.
It is well known that 
\[N_{\bR}+N_{\bH} = C'_1 \quad \mbox{ and } \quad N_{\bC} =C'_2 .\]
The second equality could be thought of
as a characterisation of $N_{\bC}$ in terms of the ``group theory'' of $G$. Brauer's 14th Problem \cite{Bra}, recently solved by Murray and Sambale \cite{MS}, is about characterising $N_{\bR}$ in 
terms of the ``group theory'' of $G$. A reader can consult other recent developments on counting conjugacy classes and
characters \cite{Mur,Rob1,Rob2}.



Now let us consider a {\em Real} structure on $G$, i.e., a group $\widehat{G}$ containing $G$ as a subgroup of index $2$ \cite{Ati}.
Given such structure, {\em Real} representations of $G$ are $\bCG$-modules where 
$\bCG = \bC \widehat{G}$ as a vector space but with a different multiplication
\[
\alpha g \ast \beta h =
\begin{cases}
  \alpha \beta gh, &  \alpha,\beta\in\bC, \; g\in G, \; h\in \widehat{G}, \\
  \alpha \overline{\beta} gh, & \alpha,\beta\in\bC, \; g\in \widehat{G}\setminus G, \; h\in \widehat{G}.
\end{cases}  
\]
The {\em Real} representations of $G$ are split into ten types according to Dyson's tenfold way \cite{Dys}.
The {\em Real} representations also correspond to characters in $\Irr$ as in Table~\ref{table_0} \cite[Table 2]{RuTa}.
Therefore,
$\Irr$ is split into ten subsets, depending on the position of the corresponding simple $\bCG$-module
within Dyson's tenfold way. On the level of multisets we have a disjoint union
\[
\cD_G = \cD_I \amalg
\cD_{I\! I} \amalg
\cD_{I\! I \! I} \amalg
\cD_{I\! V} \amalg
\cD_V \amalg
\cD_{V\! I} \amalg
\cD_{V\! I \! I} \amalg
\cD_{V\! I \! I \! I} \amalg
\cD_{I\! X} \amalg \cD_X .
\]
Let us call the cardinalities of the corresponding multisets $N_{I}, N_{I\! I}, \ldots , N_{X}$. We propose the following problem.
\begin{BDP}
 Use ``group theory'' of $\widehat{G}$ to determine all these ten  numbers as well as the corresponding multisets.
\end{BDP}

The main result of this paper is the merger of Theorem~\ref{NC_rel_2}
and Corollaries~\ref{NC_rel_5} and \ref{NC_rel_4}, where ten linearly independent ``group-theoretic'' linear constraints
for the ten numbers $N_{I}, N_{I\! I}, \ldots , N_{X}$ are produced. This solves the Brauer-Dyson Problem
for the numbers.

A reader may find it interesting to consult other recent developments in the theory of {\em Real} representations
\cite{GaYo, NoYo, NeOl, RuT2, RuYo, Yo}.

Brauer's 14th Problem and the classical real representation theory of $G$ corresponds to the trivial {\em Real} structure $\widehat{G} = G \times C_2$.
Indeed, a {\em Real} representation is a complex vector space $V$ on which the group $G$ acts linearly and the coset $\widehat{G}\setminus G$
antilinearly. In the trivial case, the coset has a canonical element $\sigma = (1,-1)$ of order 2, commuting with $G$. Hence,
the fixed points $V^\sigma$ is a $\bR G$-module such that $V\cong \bC\otimes_{\bR} V^\sigma$ as $\bCC$-modules.
Thus, we ``see'' real representations from {\em Real} representations.

A better way to look at it is to observe that $\bCC$ is isomorphic to the matrix algebra $M_2 (\bR G)$,
so we have a Morita equivalence between real representations and {\em Real} representations. 
Going back to counting the characters in $\Irr$, the Brauer-Dyson Problem degenerates to Brauer's 14th problem:
\[
\cD_\bR = \cD_I, \
\cD_\bC = \cD_{I\! V}, \ 
\cD_\bH = \cD_{V\! I \! I \! I}, \
\cD_N = \emptyset \mbox{ for } N\not\in \{ I, I\! V, V\! I \! I \! I \}.
\]

In the first section we split $\Clg$ into five subsets and produce five constraints.
In the second section
we use the result of  Murray and Sambale \cite{MS} and properties of induction to $\widehat{G}$
to produce three more constraints.
The final two constraints are found in the third section where 
we employ their method \cite{MS}.

\section*{Acknowledgements}
The first author is grateful to the Max Planck Institute for Mathematics in Bonn for its hospitality and financial support.
Also the first author thanks the Department of Mathematics of the University of Zurich for its hospitality.
The second author was supported by the departmental grant \textit{Progetto Sviluppo Dipartimentale} -- UNIPD  PSDIP23O88  at the University of Padua.

\section{Counting classes and characters in the presence of Real structure}
Let us fix a {\em Real} structure on $G$ from now on. It's a group $\widehat{G}$ containing $G$ as a subgroup of index $2$.
Let $G^\sharp \coloneqq \widehat{G}\setminus G$ be the second coset. The first observation comes directly from Table~\ref{table_0}. 
\begin{table}
\caption{Extract of \cite[Table 2]{RuTa}}
\label{table_0}
\begin{center}
\begin{tabular}{|c|c|c|c|c|c|c|c|c|c|c|}
\hline
\mbox{type of } $\chi$ & I  & II &  III &  IV &  V &  VI &  VII &  VIII &  IX &  X  \\ 
\hline
\mbox{Dyson label} & $RR$  & $QR$ &  $CR$ &  $CC2$ & $RC$ & $QC$ &  $CC1$ &  $QQ$ &  $RQ$  &  $CQ$  \\
\hline
$\bF_a$            & $\bR$ & $\bR$ & $\bR$ & $\bC$ & $\bC$ & $\bC$ & $\bC$ & $\bH$ & $\bH$ & $\bH$ \\
\hline
$\bF_d$            & $\bR$ & $\bH$ & $\bC$ & $\bC$ & $\bR$ & $\bH$ & $\bC$ & $\bH$ & $\bR$ & $\bC$ \\
\hline
$|\cA^\vee|$ & $1$ & $1$ & $2$ & $1$ & $1$ & $1$ & $2$ & $1$ & $1$ & $2$ \\
$|\cB^\vee|$ & $2$ & $1$ & $1$ & $2$ & $1$ & $1$ & $1$ & $2$ & $1$ & $1$ \\
\hline
$|\cC^\vee|$ & $1$ & $1$ & $2$ & $2$ & $2$ & $2$ & $4$ & $1$ & $1$ & $2$ \\
\hline
$|\cD^\vee|$ & $1$ & $1$ & $1$ & $1$ & $2$ & $2$ & $2$ & $1$ & $1$ & $1$ \\
\hline
\end{tabular}
\end{center}
\end{table}
The type of a complex irreducible character $\chi$ is the endomorphism ring $\bF_a$ (we follow the notation in \cite{RuTa}) of
a corresponding simple $\bR G$-module. Note that it is also recorded in the second letter of the Dyson label
with $Q$ standing for {\em the quaternions}, which we denote $\bH$.
It follows that on the level of multisets
\[ 
  \cD_{\bR} = \cD_I \amalg \cD_{I\! I} \amalg \cD_{I\! I \! I} , \
  \cD_{\bC} = \cD_{I\! V} \amalg \cD_V \amalg \cD_{V\! I} \amalg \cD_{V\! I \! I}, \ 
\cD_{\bH} = \cD_{V\! I \! I \! I} \amalg \cD_{I\! X} \amalg \cD_X .
\] 
As before,  $N_? = | \cD_? |$. 
\begin{lemma} \label{easy_facts}
  The five numbers
 $N_{I \! I \! I}$, $N_{I \! V}$, $N_{V}$, $N_{V\! I}$ and $N_{X}$ are even. 
  The number $N_{V\! I \! I}$ is divisible by 4. Moreover, 
\[ N_{\bR} = N_I + N_{I\! I} +  N_{I\! I \! I} , \
N_{\bC} = N_{I\! V} + N_V +  N_{V\! I} +  N_{V\! I \! I}, \ 
N_{\bH} = N_{V\! I \! I \! I} + N_{I\! X} + N_X .
\]
\end{lemma}
\begin{proof}
  The first two statements follow from examination of the row $|\cC^\vee|$ of Table~\ref{table_0}.
  Indeed, $\cA$, $\cB$, $\cC$ and $\cD$ are 
  the four semisimple algebras that constitute an antilinear block \cite{RuTa}.
  The table summarises the number of distinct simple modules each algebra has. The algebra $\cC$ is a direct summand
  of $\bC G$, so it draws a batch of complex irreducible characters of $G$. We see $4$ in the column ${V\! I \! I}$.
  This means that the type ${V\! I \! I}$ characters come in batches of four. Ditto for other columns.

  The three equations are a manifestation of the aforementioned fact that the multisets $\cD_{\bR}$, $\cD_{\bC}$ and $\cD_{\bH}$
  decompose into disjoint unions.
\end{proof}  

Now observe that there is an action of the Klein-four group $K_4 = \{ 1,a,b,c \}$ on the set $\Clg$ of the conjugacy classes by
\begin{equation} \label{action_classes}
K^a \coloneqq K^{-1}, \
K^b \coloneqq xKx^{-1}, \
K^c \coloneqq x(K^{-1})x^{-1},
\end{equation}
where $x\in G^\sharp$. Note that this action does not depend on the choice of $x$. We say that a class $K$ is of type $n$ if the 
$K_4$-orbit of $K$ has $n$ elements. In case of $n=2$, we will distinguish types $2a$, $2b$ and $2c$: we say $K$ is of type $2g$ if
$K^g=K$. Let us denote the number of classes of each type by $C_1, C_{2a}, C_{2b}, C_{2c}, C_{4}$ correspondingly. Clearly, $C_{2g}$ are all even
and $C_4$ is divisible by 4.

\begin{lemma} \label{NC_rel_1}
  The following three linearly independent equations connect the integers $N_?$ and $C_?$:
  \[
  \begin{cases}
    N_I + N_{I\! I} + N_{I\! I \! I} +  N_{V\! I \! I \! I} + N_{I\! X} + N_{X} & = C_1 + C_{2a} \\
    N_{I\! V} + N_V + N_{V\! I} + N_{V\! I \! I} & = C_{2b}+ C_{2c}+ C_{4} \\
    N_{I\! I \! I} + N_{I\! V}+ N_{V\! I \! I} + N_{X} & = C_{2a}+ C_{2b}+ C_{4}
\end{cases}
  \]
\end{lemma}
\begin{proof}
  The {real} classes are such classes $K\in\Clg$ that $K^a=K$. In other words, these are classes of types $1$ and $2a$. Hence, 
\[ N_{\bR} + N_{\bH} = C'_1 = C_1 + C_{2a} \ \mbox{ and } \ N_{\bC}  = C'_2 = C_{2b}+ C_{2c}+ C_{4} \, .\] 
After this, the first two equations follow from Lemma~\ref{easy_facts}. 

  Now the {\em Real} classes are unions $K\cup K^c$ for $K\in \Clg$.
So the classes of types $1$ and $2c$ are {\em Real} classes. The classes of other types join in pairs to form {\em Real} classes.
The difference between $|\Clg|$ and the number of {\em Real} classes is the number of distinct
  $\bCG$-modules with the endomorphism ring $\bC$ \cite[Cor. 5.9]{RuTa}.
This endomorphism ring is $\bF_d$ in Table~\ref{table_0} (and also the first letter of the Dyson label). 
This equality multiplied by $2$ gives  the last equation. 
To observe it, examine the rows $|\cC^\vee|$ and $|\cD^\vee|$ of Table~\ref{table_0}:
the algebra $\cD$ is a direct summand of $\bCG$,
so that
the row $|\cD^\vee|$ records the number of $\bCG$-modules in each antilinear block. 

The linear independence is clear.
\end{proof}  

There are 
two more linearly independent linear equations that allow  us to express $C_?$ from $N_?$.
To discover them, let us first contemplate the origin of the three equations in Lemma~\ref{NC_rel_1},
which is ultimately achieved in Theorem~\ref{NC_rel_2}. 
In essence, there are two centres \cite[Th. 2.5]{RuTa}
\[
Z = Z( \bC G) \geq \widehat{Z} = Z( \bC \! * \! \widehat{G}), \quad \dim_\bC Z = \dim_\bR \widehat{Z} = |\Clg|
\]
and the equations come from counting their dimensions in different ways. To dig deeper, observe that the group $K_4$ acts
on the $\bC$-algebra $Z$ by Formulas~\eqref{action_classes}.

Now define an action of $K_4 = \{ 1,a,b,c \}$ on $\Irr$
\begin{equation} \label{action_characters}
\chi^a \coloneqq \overline{\chi}, \
\chi^b \coloneqq \chi^x : z \mapsto \chi (xzx^{-1}), \
\chi^c \coloneqq \overline{\chi}^{x}
\end{equation}
and extend it linearly to the vector space of class functions $\Fun$.

\begin{lemma} The $K_4$-modules $Z$ and $\Fun$ are isomorphic.
\end{lemma}  
\begin{proof}
  The natural evaluation map
  \[ \Fun \times Z \hookrightarrow  {{{\mathrm F}{\mathrm u}{\mathrm n} (G, \bC)}} \times \bC G \rightarrow \bC \]
  is $K_4$-equivariant by the way we set the actions up. Hence, $\Fun \cong Z^\ast$. Since simple $K_4$-modules are self-dual,
  i.e., $V\cong V^\ast$ for all simple modules,  $\Fun \cong Z$.
\end{proof}

\begin{thm}\label{NC_rel_2}
  The following five linearly independent equations connect the integers $N_?$ and $C_?$:
  \[
  \begin{cases}
    N_{V\! I \! I}  & = C_4 \\
    N_{I \! I \! I} + N_{X} & = C_{2a} \\
    N_{I\! V} & = C_{2b} \\
    N_{V} + N_{V \! I} & = C_{2c} \\
    N_I + N_{I\! I} + N_{V\! I \! I \! I} + N_{I\! X} & = C_1 
    \end{cases}
    \]
\end{thm}
  \begin{proof}
    Let us denote the simple $K_4$-characters by $1$, $\lambda_a$, $\lambda_b$, $\lambda_c$. The 1 is trivial. The character $\lambda_x$
    is the non-trivial character such that $\lambda_x(x) =1$. Now both $Z$ and $\Fun$ are permutation $K_4$-modules. Let us compute their characters.
    $Z$ has a natural decomposition in terms of classes so that the character is
    \[
    C_1 1 + C_{2a}(1+\lambda_a)+ C_{2b}(1+\lambda_b)+ C_{2c}(1+\lambda_c)+ C_{4}(1+\lambda_a +\lambda_b +\lambda_c). 
    \]
    Using the properties of characters in $\Irr$, recorded in Table~\ref{table_0}, we write this permutation $K_4$-character as
    \newpage
    \[
    (N_I + N_{I\! I} + N_{V\! I \! I \! I} + N_{I\! X}) 1 + ( N_{I \! I \! I} + N_{X})(1+\lambda_a) +
    \]
    \[
    +  N_{I\! V}(1+\lambda_b)+ (N_{V} + N_{V \! I})(1+\lambda_c)+ N_{V\! I \! I}(1+\lambda_a +\lambda_b +\lambda_c). 
    \]
    This deserves a further explanation (cf. Table~\ref{table_0}).
    Each $K_4$-orbit constitutes a single antilinear block, that is, a set of simple $\cC$-modules.
    The size of this set is recorded in row $|\cC^\vee|$. When we see 1 or 4, there is nothing to decide: these are the trivial orbit and the free orbit.
    But when we see 2, it is an orbit of size 2: we must now decide whether $a$, $b$ or $c$ is in its stabiliser.
    This stabiliser controls the structure of the antilinear block and can be seen in the pattern of numbers in the four rows
    $|\cA^\vee|$, $|\cB^\vee|$, $|\cC^\vee|$ and $|\cD^\vee|$.
    It remains to observe that $\chi^a=\chi$ triggers 2 in row $|\cA^\vee|$, yielding the permutation $K_4$-character $1+\lambda_a$.
    Ditto $\chi^c=\chi$ triggers 2 in row $|\cD^\vee|$, yielding $1+\lambda_c$.
    We summarize this information in the second row of Table~\ref{table_1}.

    Comparing the coefficients at $\lambda_a$, $\lambda_b$ and $\lambda_c$, we get
    \[ N_{I \! I \! I} + N_{X}+ N_{V\! I \! I} = C_{2a}+ C_{4}, \ 
    N_{I\! V} + N_{V\! I \! I} = C_{2b}+ C_{4}, \ 
   N_{V} + N_{V \! I} + N_{V\! I \! I} = C_{2c} + C_{4}.
    \]
    Add the first two equations, subtract the third equation from Lemma~\ref{NC_rel_1}, observe that $N_{V\! I \! I}  = C_4$.

    The second, third and fourth equations follow immediately.

    Finally, use the first equation from Lemma~\ref{NC_rel_1} to prove the last equation.
  \end{proof}

  The next corollary is somewhat unexpected. It also suggests that there may be another, potentially more natural proof of Theorem~\ref{NC_rel_2}.
  \begin{cor}
    The $K_4$-sets $\Irr$ and $\Clg$ are isomorphic.
\end{cor}  
%

\section{Induction}
The main result of Murray and Sambale is a determination of $N_\bR = N_I + N_{I\! I} + N_{I \! I \! I}$ using the group theory of $G$ \cite[Th. A]{MS}.
This gives one additional equation (already written in Lemma~\ref{easy_facts}) on $N_?$ on top of those in Theorem~\ref{NC_rel_2}.
To derive more, consider the restriction-induction relation between $\Irr$ and $\Irh$.
Let us recall the three Frobenius-Schur indicators in play here \cite{RuTa}
\[\cF,{\cF^{\sharp}} : \Irr \rightarrow \{-1,0,1\}, \quad
\cF (\chi) = \frac{1}{|G|} \sum_{g\in G} \chi (g^2), \quad
{\cF^{\sharp}} (\chi) = \frac{1}{|G|} \sum_{g\in G^{\sharp}} \chi (g^2),\]
\[\widehat{\cF} : \Irh \rightarrow \{-1,0,1\}, \quad
\widehat{\cF} (\psi) = \frac{1}{|\widehat{G}|} \sum_{g\in \widehat{G}} \psi (g^2).\]
We say that $\chi\in\Irr$ and $\psi\in\Irh$ correspond to each other, if $\chi$ is a constituent of $\psi|_{G}$.
We write it as $\chi \leftrightarrow \psi$.
By $\epsilon\in\Irh$ we denote the linear sign character
\[\epsilon: \widehat{G} \rightarrow \widehat{G}/G \cong C_2 \hookrightarrow \bC^{\times} \, .\] 
The next well-known lemma follows from Mackey's Formula.
\begin{lemma} \label{ind}
  Suppose $\chi\leftrightarrow \psi$.  
  The following statements hold.
  \begin{enumerate}
  \item If $\chi^b =\chi$, then $\chi = \psi\!\mid_{G}$. Futhermore, the correspondence is $1\! :\! 2$, that is, 
    $\chi \leftrightarrow \psi$, $\chi \leftrightarrow \epsilon \otimes \psi$ and $\psi \neq \epsilon \otimes \psi$.
    Finally, 
    $\widehat{\cF}(\psi) = ({\cF} (\chi) + {\cF^\sharp}(\chi))/2$.
  \item If $\chi^b \neq \chi$, then $\chi +\chi^b = \psi\!\mid_{G}$.
Futhermore, the correspondence is $2\! : \! 1$, that is, 
$\chi\leftrightarrow \psi$, $\chi^b \leftrightarrow \psi$ and $\psi = \epsilon \otimes \psi$.
Finally, $\widehat{\cF}(\psi) = {\cF} (\chi) + {\cF^\sharp}(\chi)$.
  \end{enumerate}
\end{lemma}
Table~\ref{table_1} is a summary of how Lemma~\ref{ind} works in all the ten Dyson types, cf. \cite[Th 4.2]{RuTa}.
The letters Y and N stand for Yes and No. The table also includes the stabiliser of $\chi$ under the $K_4$-actions,
cf. the proof of Theorem~\ref{NC_rel_2}.
\begin{table}
\caption{Indicators and the correspondence $\chi\leftrightarrow \psi$}
\label{table_1}
\begin{center}
\begin{tabular}{|c|c|c|c|c|c|c|c|c|c|c|}
\hline
\mbox{type of } $\chi$ & I  & II &  III &  IV &  V &  VI &  VII &  VIII &  IX &  X  \\ 
\hline
$\mbox{Stab}_{K_4}(\chi)$ & $K_4$  & $K_4$  &  $ \langle a \rangle$ &  $\langle b \rangle$ &  $\langle c \rangle$ &  $\langle c \rangle$ &  $\{1\}$ &  $K_4$  &  $K_4$  &  $\langle a \rangle$ \\ 
$\chi^b=\chi$  & Y & Y & N & Y & N & N & N & Y & Y & N \\
\hline
$\cF (\chi)$ & $1$ & $1$ & $1$ & $0$ & $0$ & $0$ & $0$ & $-1$ & $-1$ & $-1$ \\
$\cF^\sharp (\chi)$ & $1$ & $-1$ & $0$ & $0$ & $1$ & $-1$ & $0$ & $-1$ & $1$ & $0$ \\
$\widehat{\cF} (\psi)$ & $1$ & $0$ & $1$ & $0$ & $1$ & $-1$ & $0$ & $-1$ & $0$ & $-1$ \\
\hline
\end{tabular}
\end{center}
\end{table}
The next lemma follows immediately from Table~\ref{table_1} and Lemma~\ref{ind}.
\begin{lemma}\label{NC_rel_5a}
  Let $\widehat{N}_{\bR}$, $\widehat{N}_{\bC}$ and $\widehat{N}_{\bH}$ be the number of characters of each type in $\Irh$.
  We have the following three linear equations. 
  \[
  \begin{cases}
        2N_I + \frac{1}{2} N_{I\! I \! I} + \frac{1}{2} N_{V} & = \widehat{N}_{\bR} \\
    2N_{I\! I} + 2N_{I \! V} +\frac{1}{2} N_{V\! I \! I} + 2N_{I \! X} & = \widehat{N}_{\bC} \\
    \frac{1}{2} N_{V\! I} + 2N_{V\! I \! I \! I} + \frac{1}{2} N_{X} & = \widehat{N}_{\bH}
    \end{cases}
  \]
\end{lemma}

Notice that
$\widehat{N}_{\bC} = \widehat{C}_2$ is determined by counting conjugacy classes in $\widehat{G}$.
Then ${N}_{\bR}$ and $\widehat{N}_{\bR}$ are determined as in \cite{MS}.
The next corollary follows immediately from
Lemma~\ref{NC_rel_5a} together with Theorem~\ref{NC_rel_2} and Lemma~\ref{easy_facts}. 
\begin{cor}\label{NC_rel_5}
  The following three equations are linearly independent of the equations in Theorem~\ref{NC_rel_2}, bringing the total number of linearly independent equations to eight:
    \[
  \begin{cases}
N_I + N_{I\! I}  + N_{I\! I \! I} & = {N}_{\bR} \\
    4N_I + N_{I\! I \! I} + N_{V} & = 2 \widehat{N}_{\bR} \\
    N_{I\! I} + N_{I \! X} & = \frac{1}{2} \widehat{C}_{2} - C_{2b} - \frac{1}{4} C_4 
\end{cases}
  \]
\end{cor}

\section{Equations in groups}
First we need a recognition result for a multiset $X=[q_1, \ldots , q_n]$ of cardinality $n$, consisting of real numbers.
Consider the power-sum functions $\varpi_k (X) \coloneqq \sum_{q\in X} q^k$.
\begin{lemma} \label{powers} (cf. \cite[Lemma 1]{MS})
A real multiset $X$ of cardinality $n$ is uniquely determined by the $n$ values $\varpi_k (X)$ for $k=1,\ldots ,n$.
\end{lemma}
\begin{proof}
  The statement follows from Newton's Theorem about symmetric functions, i.e., that
  $\bQ[q_1, \ldots, q_n]^{S_n} = \bQ[\varpi_1, \ldots, \varpi_n]$,  
\end{proof}

We believe in the following conjecture.
\begin{con} \label{powers_1}
A real multiset $X$ of cardinality $n$ is uniquely determined by the $n$ values $\varpi_{3k-2} (X)$ for $k=1, 2, \ldots, n$.
\end{con}

Notice that Conjecture~\ref{powers_1} holds for a generic $X$ because the fields of rational functions
$\bQ (\varpi_1, \varpi_2, \ldots, \varpi_n)$ and
$\bQ (\varpi_1, \varpi_4, \ldots, \varpi_{3n-2})$  are equal \cite{DvZ}.
In general, recognising a multiset from power-sums is an active area of research, cf. \cite{MSW}.
Here we prove only a weaker version of the conjecture, sufficient for our ends.

\begin{lemma} \label{powers_2}
  Let $Y$ be a finite set of real numbers. Consider only multisets that consist of elements of $Y$.
  Such a multiset $X$ of cardinality $n$ is uniquely determined among all other such multisets
  by finitely many values $\varpi_{3k-2} (X)$ for $k=1, 2, \ldots$
\end{lemma}
\begin{proof}
  Consider a pair $X,X'$ of such multisets.
Pick $q\in X$ such that $|t|\leq |q|$ for all $t\in X$.
For large enough $k$, the power sum $\varpi_{3k-2} (X)$ is dominated by the terms $(\pm q)^k$.

Ditto for $X'$. It follows that examining finitely many $\varpi_{3k-2} (X)$ and $\varpi_{3k-2} (X')$ allows us to conclude
whether $X$ and $X'$ contain the same number of $q$ and $-q$.
Moreover, examining finitely many $\varpi_{3k-2} (X)$ and $\varpi_{3k-2} (X')$ allows us
to decide whether $X=X'$.

Since there are only finitely many possible $X$, the lemma is proved.
\end{proof}

Consider a word $w$ in the free group $F= F\langle y_1, y_2, \ldots z_1, z_2, \ldots \rangle$ on two sets of variables.
The word $w$ determines a $\bN$-valued class-function $\Theta_w\in \Fun$, where $\Theta_w (g)=\Theta (w ,g)$ is the number
of solutions of the equation $g = w( y_1, \ldots z_1, \ldots )$ with $y_i\in G$ and $z_i\in G^\sharp$.

\begin{thm} \label{NC_rel_3} Let us consider two sequences of words in $F$
  \[
  v_m = z_1^2 z_2^2 \cdots z_m^2,  \ m=3,4,\ldots \ w_n = y_1^2y_2^2z_1^2 y_3^2y_4^2 \cdots y_{2n}^2 z_n^2, \ n=1,2,\ldots
  \]
  The following statements hold.
  \begin{enumerate}
%
  \item
    The first $1+|\Clg|$ elements of the sequence $\Theta( v_m, 1)$ determine the numbers
    $N_I + N_{V} + N_{I \! X}$
    and 
    $N_{I \! I} + N_{V\! I} + N_{V \! I \! I \! I}$.
  \item
    Finitely many elements of the sequence $\Theta( w_n, 1)$ determine the numbers
    $N_I +     N_{I\! I}$
    and
    $N_{V \! I \! I \! I} + N_{I \! X}$.
\end{enumerate}    
\end{thm}
\begin{proof}
    Consider the two multisets
    \[
    X_v = \left[ \frac{{\cF^{\sharp}}(\chi)}{\chi(1)} \,\mid\, \chi\in \Irr \right], \
    X_w = \left[ \frac{\cF(\chi)^2{\cF^{\sharp}}(\chi)}{\chi(1)} \,\mid\, \chi\in \Irr \right].
    \]
    The key observation is a formula for the power sums for these multisets: 
    \begin{equation} \label{key}
    \varpi_{m-2} (X_v) = \frac{\Theta(v_m, 1)}{|G|^{m-1}} , \  \varpi_{3n-2} (X_w) = \frac{\Theta(w_n, 1)}{|G|^{3n-1}}\, .
    \end{equation}
Indeed, let us consider the central elements 
\[ S \coloneqq \frac{1}{|G|} \sum_{g\in G} g^2, \ S_1 \coloneqq \frac{1}{|G|} \sum_{g\in G^{\sharp}} g^2 , \ S_2 \coloneqq S^2 S_1 \ \in Z \; .\]
Observe that
\[
S_1^n = \sum_{g\in G} \frac{\Theta( v_n, g)}{|G|^n} g, \
S_2^n = \sum_{g\in G} \frac{\Theta( w_n, g)}{|G|^{3n}} g.
\]
Let $\widehat{\chi}:Z \rightarrow \bC$ be the homomorphism defined by $\chi$, i.e., $\widehat{\chi} (\sum_g a_g g) = \sum_g a_g \chi (g)/\chi (1)$.
Note that
\[
\widehat{\chi} (S_2^n) = (\widehat{\chi} (S)^2 \widehat{\chi} (S_1))^n =
\left( \frac{\cF(\chi)^2{\cF^{\sharp}}(\chi)}{\chi(1)^3}\right)^n \, . 
\]
Now we use the orthogonality relations:
  \begin{align*}
  	 \frac{\Theta( w_n, 1)}{|G|^{3n-1}} &= \sum_{g\in G} \frac{\Theta( w_n, g)}{|G|^{3n}} \sum_{\chi\in\Irr} \chi(g)\chi(1)  =  \sum_{\chi\in\Irr} \chi(1)^2 \sum_{g\in G} \frac{\Theta( w_n, g)}{|G|^{3n}} \frac{\chi(g)}{\chi(1)} = \\
&= \sum_{\chi\in\Irr} \chi(1)^2 \widehat{\chi} (S_2^n)  = 
   \sum_{\chi\in\Irr} \chi(1)^2 \left( \frac{\cF(\chi)^2{\cF^{\sharp}}(\chi)}{\chi(1)^3} \right)^n = \sum_{q\in X_w} q^{3n-2} \, .
  \end{align*}
  Similarly, one proves~\eqref{key} 
  for $X_v$, using the element $S_1$.

  It remains to examine Table~\ref{table_1} 
  to determine what terms $q^k$ we see in the power sums $\varpi_{m-2} (X_v)$ and $\varpi_{3n-2} (X_w)$.
  The numerators are always 0 or $\pm 1$, while the denominators are $\chi (1)^k$:
  \begin{enumerate}
      \item In $X_v$ the characters $\chi$ of types $I$, $V$ and $I\! X$ contribute positive numbers;
    those of types $I\! I$, $V\! I$ and $V\! I\! I\! I$ contribute negative numbers;
    the rest contribute zeroes.
      \item In $X_w$ the characters $\chi$ of types $I$ and $I\! X$ contribute positive numbers;
    those of types $I\! I$ and $V\! I\! I\! I$  contribute negative numbers;
    the rest contribute zeroes. Also the multiset $X_w$ contains only elements from the finite set
    $\{0\} \cup \{\pm 1/d \,\mid\, d \mbox{ divides } |G|\}$. 
  \end{enumerate}
  Therefore, by Lemma~\ref{powers} and  Lemma~\ref{powers_2}, the sequences allow us to count the numbers of positive and negative numbers in each multiset. 
  \end{proof}

Theorem~\ref{NC_rel_3} supplies the two final linearly independent equations.
Let $S_v$ and $S_w$ be the number of positive numbers in the multiset $X_v$ and $X_w$, used in the proof just above.
\begin{cor}\label{NC_rel_4}
  The following two equations are linearly independent of the equations in Theorem~\ref{NC_rel_2} and Corollary~\ref{NC_rel_5},
  bringing the total number of linearly independent equations to ten:
    \[
  \begin{cases}
    N_{I} + N_V + N_{I \! X} & = S_v \\
    N_I + N_{I \! X} & = S_w \\
\end{cases}
  \]
  \end{cor}

\end{document}